\documentclass{amsart}
\usepackage{amsmath, amsfonts, amssymb,amsthm}
\usepackage{amstext}
\usepackage{mathrsfs}
\usepackage[dvipsnames]{xcolor}
\usepackage{tikz}
\usepackage[all]{xy}
\usepackage{enumerate}
\usetikzlibrary{er,positioning}

\usepackage{lineno}

\usepackage{comment}

\theoremstyle{definition}
\newtheorem{definition}{Definition}[section]

\newtheorem{theorem}[definition]{Theorem}

\newtheorem{corollary}[definition]{Corollary}
\newtheorem{remark}[definition]{Remark}

\newcommand{\sq}[1]{\ifx#1([\else\ifx#1)]%
  \else\message{invalid use of "sq"}\fi\fi}

\DeclareMathSymbol{\idot}{\mathbin}{operators}{`\.}
\allowdisplaybreaks
\begin{document}
\title{The Second Main Theorem with moving hypersurfaces in subgeneral position}
\author{Qili Cai}
\address{Department of Mathematics\newline
\indent University of Houston\newline
\indent Houston,  TX 77204, U.S.A.} 
\email{qcai3@cougarnet.uh.edu}
\author{Chin Jui Yang}
\address{Department of Mathematics\newline
	\indent University of Houston\newline
	\indent Houston,  TX 77204, U.S.A.} 
\email{cyang36@cougarnet.uh.edu}

\begin{abstract}
	In this paper, we prove a second main theorem for a holomorphic curve $f$ into $\mathbb P^N (\mathbb C)$ with a family of slowly moving hypersurfaces $D_1,...,D_q$ with respect to $f$ in $m$-subgeneral position, proving an inequality with factor $3 \over 2$. The motivation comes from the recent result of Heier and Levin.
	
\end{abstract}
 \thanks{2020\ {\it Mathematics Subject Classification.}
  32H30, 32A22.}  

\baselineskip=16truept \maketitle \pagestyle{myheadings}
\markboth{}{Non-integrated defect}
\section{introduction}

In 1929, relating to the study of value distribution theory for meromorphic functions, R. Nevanlinna \cite{Nev} conjectured that the second main theorem for meromorphic functions is still valid if one replaces the fixed points 
 by meromorphic functions of slow growth. This conjecture was solved by Osgood\cite{CF}, Steinmetz\cite{Stein} and Yamanoi\cite{Yama} with truncation one. In 1991, Ru and Stoll \cite{RuS2} established the second main theorem for linearly nondegenerate holomorphic curves and moving hyperplanes in subgeneral position.
  
\begin{theorem}[Ru and Stoll \cite{RuS2}] Let $f: \Bbb C\rightarrow \Bbb P^N (\Bbb C)$ be a holomorphic map, and let ${\mathcal H}:=\{H_1, ... ,H_q\}$ be a family of slowly moving hyperplanes with respect to $ f $ located in 
	$m$-subgeneral position.  Assume that $f$ is linearly nondegenerate over ${\mathcal K}_{{\mathcal H}}$.  Then, for any $\epsilon > 0$,
	$$\sum\limits_{j = 1}^q m_f( r,H_j) \leq_{exc} (2m-N+1 + \epsilon)T_f( r ),$$
	where “$\leq_{exc}$” means that the above inequality holds for all $ r $ outside a set with finite Lebesgue measure.
\end{theorem}  
  
Before stating our main result, we recall some basic definitions for moving targets. We say a meromorphic function $g$ on $\mathbb C$ is of slow growth with respect to $f$ if $T_g( r ) = o(T_f( r ))$. Let $\mathcal K_f$ be the field of all meromorphic functions on $\mathbb C$ of slow growth with respect to $f$ which is a subfield of meromorphic functions on $\mathbb C$. For a positive integer $d$, we set 
$${{\mathcal I}_d}: = \left\{ I = ( i_0,...,i_N ) \in {\Bbb Z}_{\geq 0}^{N + 1} \ \big| \ {i_0} + \cdots + {i_N} = d  \right\}$$ 
and
$${n_d} = \# {\mathcal I_d} = \left( {\begin{array}{*{20}{c}}
	{d + N}  \\
	N  \
	
	\end{array} } \right).$$
A {\it moving hypersurface} $D$ in $\mathbb P^N( \mathbb C )$ of degree $d$ is defined by a homogeneous polynomial $Q = \sum\nolimits_{I \in {\mathcal I_d}} {{a_I}{{\bf x}^I}}$, where $a_I, I \in \mathcal I_d$, are holomorphic functions on $\mathbb C$ without common zeros, and ${\bf x}^I={x_0}^{i_0}\cdots{x_N}^{i_N}$.   Note that $D$ can be regarded as a holomorphic map $a:\Bbb C \to {\Bbb P}^{{n_d} - 1}( \Bbb C)$ with a reduced representation $( { ... ,{a_I}(z), ... } )_{I \in \mathcal I_d}$. We call $D$ a {\it slowly moving hypersurface} with respect to $f$ if $T_a( r ) = o(T_f( r ))$.

\begin{definition}
Under the above notations, we say that $f$ is {\it linearly nondegenerate over} $\mathcal K_f$ if there is no nonzero linear form $L \in \mathcal K_f \left[ x_0,...,x_N \right]$ such that $L(f_0,...,f_N) \equiv 0$, and $f$ is {\it algebraically nondegenerate over} $\mathcal K_f$ if there is no nonzero homogeneous polynomial $Q \in \mathcal K_f[x_0,...,x_N]$ such that $Q(f_0,...,f_N) \equiv 0$. 
\end{definition}

\begin{remark}\label{R1}
In this paper, we only consider those moving hypersurfaces $D$ with defining function $Q$ such that $Q(f_0,...,f_N) \not\equiv 0$.
\end{remark}

We say that the moving hypersurfaces $D_1, \dots,D_q$ are in $m$-{\it subgeneral position} if there exists $z \in \mathbb C$ such that $D_1(z), \dots, D_q(z)$ are in $m$-subgeneral position (as fixed hypersurfaces), i.e., any $m+1$ of $D_1(z),...,D_q(z)$ do not meet at one point. Actually, if the condition is satisfied for one point $z \in \mathbb C$, it is also satisfied for all $z \in \mathbb C$ except for a discrete set.

In this paper, we consider the moving hypersurfaces in $m$-subgeneral position and prove the following theorem. The method of proving our main theorem is motivated by the recent result of Heier-Levin \cite{HL}. 
 \begin{theorem}[Main Theorem] \label {main}
 Let $f: \Bbb C\rightarrow \Bbb P^N (\Bbb C)$ be a holomorphic map, and let $D_1, \ldots ,D_q$ be a family of slowly moving hypersurfaces  with respect to $ f $ of degree $d_1, \dots, d_q$, respectively. Assume that $f$ is algebraically nondegenerate over $\mathcal K_f$ and $D_1, \dots, D_q$ are located in $m$-subgeneral position. Then, for any $\epsilon > 0$,
	$$\sum\limits_{j = 1}^q {1\over d_j} m_f( r,D_j) \leq_{exc} {3\over 2}(2m-N+1 + \epsilon)T_f( r ).$$
	Here “$\leq_{exc}$” means that the above inequality holds for all $ r $ outside a set with finite Lebesgue measure.
\end{theorem}

Indeed, we prove a more general case when $f$ is degenerate over $\mathcal K_f$.  To do so, we introduce the notion of ``universal fields". Let $k$ be a field. The {\it universal field} $\Omega_k$ of $k$ is a field extension of $k$ which is algebraically closed and has infinite transcendence degree over $k$. In this paper, we take $k=\mathcal K_f$ and denoted by $\Omega$ the universal field over $k=\mathcal K_f$. Let $f=[f_0:f_1:\cdots:f_N]$ be a reduced representation of $f$. We can regard each $f_i$, $0 \leq i \leq N$, as an element in $\Omega$. Hence $f$ can be seen as a set of homogeneous coordinates of some point $P$ in $\mathbb P ^N(\Omega)$. Equip $\mathbb P ^N(\Omega)$ with the natural Zariski topology. Let $V_f$ be the closure of $P$ in $\mathbb P ^N(\Omega)$ over $\mathcal K _f$, i.e., 
\begin{equation}\label{Vf}
	V_f: =  \cap \left\{ {P = 0\left| {P \in \mathcal K_f\left[ x_0, \ldots ,x_N \right],P(f) \equiv 0} \right.} \right\} \subseteq \mathbb P^N( \Omega).
\end{equation}
Note that $f$ is algebraically nondegenerate over $\mathcal K_f$ is equivalent to $V_f=\mathbb P ^N(\Omega)$. 
We also note that every moving hypersurface $D$ with defining function $Q\in \mathcal{K}_f[x_0, \dots, x_N]$ can be seen as a hypersurface determined by $Q$ in $\mathbb{P}^N(\Omega).$

Let $V \subset {\Bbb P}^N(\Omega)$ be an algebraic subvariety and $D_1,...,D_q$ be $q$ hypersurfaces in ${\Bbb P}^N(\Omega)$. We say that $D_1,...,D_q$ are {\it in} $m$-{\it subgeneral position on} $V$ if for any $J \subset \{1,\cdots,q\}$ with $\text{\#}J \leq m+1$,
$$\dim{ \cap _{j \in J}}{D_j} \cap V \leq m - \text{\#}J.$$
When $m=n$, we say  $D_1,\cdots,D_q$ are {\it in general position on} $V$.
Note that
$\dim{ \cap _{j \in J}}{D_j}(z) \cap V(z) = \dim{ \cap _{j \in J}}{D_j} \cap V $
for all $z \in \mathbb C$ excluding a discrete subset.

\begin{remark}
	By Lemma 3.3 in \cite{yan}, the definition of $m$-subgeneral position above implies the definition of $m$-subgeneral position below Remark \ref{R1}.
	
\end{remark}

We prove the following general result. 
\begin{theorem}\label{T13}  
	Let $f$ be a holomorphic map of ${\Bbb C}$ into ${\Bbb P}^N ({\Bbb C})$. Let $\mathcal D=\{D_1, \ldots ,D_q\}$ be a family of slowly moving hypersurfaces  in $\mathbb P ^N(\mathbb C)$ with respect to $f$ with $\deg D_j =d_j (1\leq j \leq q$). Let $V_f\subset {\Bbb P}^N(\Omega)$ be given as in (\ref{Vf}). Assume that $D_1, \dots, D_q$ are in $m$-subgeneral position on $V_f$ and $\dim V_f=n$. Assume that the following 
	Bezout property holds on $V_f$ for intersections among the divisors: If $I, J \subset \{1, ... , q\}$ then
	$$ \mbox{codim}_{V_f} D_{I\cup J}= \mbox{codim}_{V_f}(D_I\cap D_J)\leq   \mbox{codim}_{V_f}D_I+ \mbox{codim}_{V_f}D_J,$$
where for every subvariety $Z$ of $\mathbb{P}^N(\Omega),$ $\mbox{codim}_{V_f}Z$ is given by $\mbox{codim}_{V_f}Z:= \mbox{dim}V_f - \mbox{dim}V_f\cap Z. $
	Then $$\sum_{j=1}^q {1\over d_j} m_f(r, D_j)\leq_{exc} {3\over 2}(2m-n+1+\epsilon)T_f(r).$$
\end{theorem}

\begin{remark}   
	Recall that we only consider those moving hypersurfaces $D$ with defining function $Q \in \mathcal K_f [x_0,...,x_N]$ such that $Q(f) \not\equiv 0$. So we have $V_f \not\subset D_j$ for every $ 1 \leq j \leq q$.
\end{remark}

It is known that the Bezout property holds on projective spaces. Therefore Theorem \ref{main} is the special case  of Theorem \ref{T13} when $V_f={\Bbb P}^N(\Omega)$.  Therefore the rest part of the paper is devoted to prove Theorem \ref{T13}.

\section{Proof of Theorem \ref{T13}}
We begin with some background materials about moving targets. Let $f: \mathbb C \to \mathbb P^N (\mathbb C)$ be a holomorphic map. The characteristic function of $f$ is defined by 
$${T_f}( r ) = \int_0^{2\pi } {\log \left\| {{\bf f}( re^{i\theta }}) \right\|} \frac{d\theta }
{2\pi },$$
where ${\bf f} = (f_0,...,f_N)$ is a reduced representation of $f$ with entire functions $f_0,...,f_N$ having no common zeros and $\left\| {{\bf f}( z )} \right\| = \max \left\{ {\left| {f_0(z)} \right|,...,\left| {f_N(z)} \right|} \right\}$. 
The {\it proximity function} of $f$ with respect to the moving hypersurface $D$ defined by a homogeneous polynomial $Q$ is defined by 
$$m_f(r,D) = \int_0^{2\pi } \lambda _{D( re^{i\theta })}({\bf f}( re^{i\theta } ))\frac{d\theta }{2\pi }$$
where $\lambda _{D(z)}({\bf f}(z)) = \log \frac{{{\left\| {{\bf f}( z )} \right\|}^d}\left\| {Q( z )} \right\|}
{\left\| {Q( {\bf f} )( z )} \right\|}$ is the  Weil function associated to $D$ composites with $f$  and $\left\| {Q( z )} \right\| = \mathop {\max }\limits_{I \in \mathcal I_d} \left\{ {\left| {a_I( z )} \right|} \right\}$. If $D$ is a slowly moving hypersurface with respect to $f$ of degree $d$, we have 
$$m_f(r,D) \leq dT_f(r)+o(T_f(r))$$
by the first main theorem for moving targets.

In 2022, Quang \cite{Qua22} introduced the notion of distributive constant $\Delta$ as follows:
\begin{definition}\label{D21}
	Let $f: \mathbb C \to \mathbb P^N (\mathbb C)$ be a holomorphic curve. Let $D_1,\cdots,D_q$ be $q$ hypersurfaces in ${\Bbb P}^N(\Omega)$. Let $V_f$ be given as in (\ref{Vf}). We define the distributive constant for $D_1, ... ,D_q$ with respect to $f$ by 
	$$\Delta: = \max_{\Gamma\subset \{1, ... , q\}}{\#\Gamma\over \mbox{codim}_{V_f} (\bigcap_{j\in \Gamma}D_j)}.$$
\end{definition}

We remark that Quang's original definition (See Definition 3.3 in \cite{Qua22}) is different from Definition \ref{D21}. But by Lemma 3.3 in \cite{yan}, we can see that Definition \ref{D21} is equivalent to Definition 3.3 in \cite{Qua22}. We re-phrase the definition, according to Heier-Levin \cite{HL}, as follows:

\begin{definition}\label{D22}
	With the assumptions and notations in Definition \ref{D21}, for a closed subset $W$ of $V_f$, let
	$$\alpha(W)=\#\{j~|~W\subset \mbox{Supp}D_j\}.$$
	We define
	$$\Delta: =\max_{\emptyset \subsetneq W\subsetneq V_f}{\alpha(W)\over \mbox{codim}_{V_f} W}.$$
\end{definition}

We show that the above two definitions are equivalent. Suppose that $\tilde{W}$ is a subvariety of $V_f$ such that ${\alpha(W)\over \mbox{codim}_{V_f} W}$ attains the maximum at $W = \tilde{W}.$ Reordering if necessary, we assume that $\tilde{W} \subset D_j$ for $j = 1, \dots, \alpha(\tilde{W}).$ Let $W' = \cap_{j=1}^{\alpha(\tilde{W})}D_j.$ Then, clearly, $\tilde{W} \subset W'$ and hence $\mbox{codim}_{V_f} \tilde{W} \geq \mbox{codim}_{V_f} W'.$ On the other hand, $\tilde{W} \nsubseteq D_j$ for all $j > \alpha(\tilde{W})$ implies that $W' \nsubseteq D_j$ for all $j > \alpha(\tilde{W})$. So $\alpha(\tilde{W}) = \alpha(W').$ Thus we have 
$$
{\alpha(W')\over \mbox{codim}_{V_f} W'} \geq {\alpha(\tilde{W})\over \mbox{codim}_{V_f} \tilde{W}}.
$$
By our assumption for $\tilde{W}$, we get
$$
{\alpha(W')\over \mbox{codim}_{V_f} W'} = {\alpha(\tilde{W})\over \mbox{codim}_{V_f} \tilde{W}}.
$$
This means that, in Definition \ref{D22}, we only need to consider those $W$ which are the intersections of some $D_i$'s, and our claim follows from this observation. In the following, when we deal with $W$, we always assume that $W$ is the intersection of some $D_i$'s.

S.D. Quang obtained the following result.

\begin{theorem}[S.D. Quang, Lei Shi, Qiming Yan and Guangsheng Yu \cite{SYY}]\label{Qua33}  Let $f$ be a holomorphic map of ${\Bbb C}$  into ${\Bbb P}^N ({\Bbb C})$. 
	Let $\{D_j\}_{j=1}^q$ be a family of slowly moving hypersurfaces in ${\Bbb P}^N({\Bbb C})$ with $\deg D_j =d_j (1\leq j \leq q$). 
	Let $V_f \subset \mathbb P^N(\Omega)$ be given as in (\ref{Vf}). Assume that $\dim V_f=n$.
	Then, for any $\epsilon>0$, 
	$$\sum_{j=1}^q {1\over d_j} m_f(r, D_j)\leq_{exc} \left((n+1)\max_{\emptyset \subsetneq W\subsetneq V_f}{\alpha(W)\over \mbox{codim}_{V_f} W}+\epsilon\right)T_f(r).$$
\end{theorem} 
We derive the following corollary of the Theorem \ref{Qua33}. 
   \begin{corollary}\label{cor}We adopt the assumption in Theorem \ref{Qua33}. Let $W_0$ be a closed subset of $V_f\subset {\Bbb P}^N(\Omega)$. Then, for any $\epsilon > 0,$ we have
   $$\sum_{j=1}^q {1\over d_j} m_f(r, D_j)\leq_{exc} \left(\alpha(W_0)+(n+1)\max_{\phi  \subsetneq W \subsetneq {V_f}} {\alpha(W)-\alpha(W\cup W_0)\over  \mbox{codim}_{V_f} W}+\epsilon\right)T_{f}(r).$$
   \end{corollary}
   \begin{proof} 
   	Without loss of generality, we suppose that $W_0 \subset  \mbox{Supp} D_j$ for $j = q-\alpha(W_0)+1, \dots, q$.  Let $q'=q-\alpha(W_0)$. Let 
   $$\alpha'(W)=\#\{i\leq q'~|~W\subset     \mbox{Supp}D_i\}.$$
   Note that  $\alpha'(W) = \alpha(W)-\alpha(W \cup W_0)$. 
   Then by the first main theorem for moving targets and $\alpha(W_0)=q-q'$,
    $$\sum_{i=q'+1}^q {1\over d_i} m_f(r, D_i)\leq (\alpha(W_0)+\epsilon)T_{f}(r).$$
 Thus, Theorem \ref{Qua33} implies that 
 \begin{eqnarray*}\sum_{j=1}^q {1\over d_j} m_f(r, D_j)&=&  \sum_{i=1}^{q'} {1\over d_i} m_f(r, D_i) +\sum_{i=q'+1}^q {1\over d_i} m_f(r, D_i)   \\
 &\leq_{exc}& \left((n+1)\max_{\emptyset \subsetneq W\subsetneq V_f}{\alpha'(W)\over \mbox{codim}_{V_f} W}+\epsilon\right) T_{f}(r) \\
 &~&\\
 &~& + \Big(\alpha(W_0)+\epsilon\Big)T_{f}(r)
 \end{eqnarray*}
 which is we desired.
 \end{proof}
 
 \noindent{\it Proof of Theorem \ref{T13}}.
 	
\begin{proof} 
	We divide the proof into two cases. The first case is that for every algebraic subvariety $W \subset V_f\subset {\Bbb P}^N(\Omega)$ with $W\not=\emptyset$, we have $\mbox{codim}_{V_f} W \ge {n+1\over 2m-n+1}\alpha(W)$. Then, by Definition \ref{D22}, we have $\Delta \leq {2m-n+1 \over n+1}$. So the result  follows easily from Theorem \ref{Qua33}. 
	
	 Otherwise, we take a subvariety $W_0 \subset V$ such that the quantity
 $${n+1-\mbox{codim}_{V_f} W\over 2m-n+1-\alpha(W)}$$
 is maximized at $W = W_0$. We assume that $W_0$ is an intersection 
 $D_I$ for some $I \subset \{1, \cdots , q\}$. Let 
 $$\sigma:={n+1-\mbox{codim}_{V_f} W_0\over 2m-n+1-\alpha(W_0)}.$$ 
 Note that $\sigma$ is the slope of the straight line passing through $(2m-n+1,n+1)$ and $(\alpha(W_0), \mbox{codim}_{V_f}W_0)$.  
 
 Take arbitrary $\emptyset \subsetneq  W \subsetneq V_f$. By Corollary \ref{cor},  it suffices to show that 
 $$ \alpha(W_0)+(n+1){\alpha(W)-\alpha(W\cup W_0)\over  \mbox{codim}_{V_f} W}\leq {3\over 2}(2m-n+1).$$  
  
  Assume that $W = D_J$ for some nonempty $J \subset \{1,...,q\}$  (the case $\alpha(W ) = 0, J = \emptyset$, follows from $m$-subgeneral position).
  From the claim in page 19 of Heier-Levin \cite{HL} (apply the same argument in \cite{HL}), 
  we have $${\alpha(W_0)-\alpha(W\cup W_0)\over \mbox{codim}_{V_f} W}\leq {1\over \sigma}.$$
    Hence
    \begin{equation}\label{new}\alpha(W_0)+(n+1){\alpha(W)-\alpha(W\cup W_0)\over  \mbox{codim}_{V_f} W}\leq \alpha(W_0)+{n+1\over \sigma}.\end{equation}
    Finally, consider Vojta's Nochka-weight-diagram \cite{Voj07} (See the figure below).
    
    
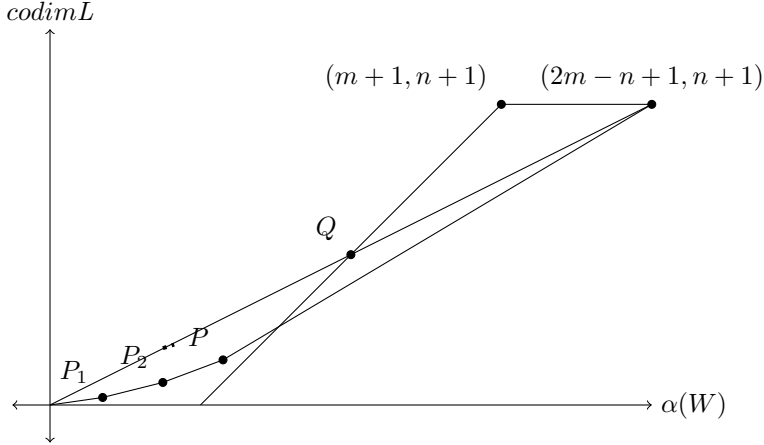
\begin{tikzpicture}
    \draw[<->] (-0.5, 0) -- (8, 0) node[right] {$\alpha(W)$};
    \draw[<->] (0, -0.5) -- (0, 5) node[above] {$codim \textit{L}$};
    \draw[scale=2, domain=1:3, smooth, variable=\x, black] plot ({\x}, {\x - 1});
    \draw[scale=2, domain=0:4, smooth, variable=\x, black] plot ({\x}, {0.5 * \x});
    \draw[scale=2, domain=3:4, smooth, variable=\x, black] plot ({\x}, {2});
    % create point
    \filldraw (4,2) circle[radius=1.5pt] node[above left, outer sep=2pt] {$Q$};
    
    \filldraw (0.7,0.1) circle[radius=1.5pt] node[above left, outer sep=2pt] (p1) {$P_1$};
    \filldraw (1.5,0.3) circle[radius=1.5pt] node[above left, outer sep=2pt] (p2) {$P_2$};
    \filldraw (2.3,0.6) circle[radius=1.5pt] node[above left, outer sep=2pt] (p3) {$P$};
    
    % p 2 p
    \draw[dotted, line width=1.5pt] (p2) -- (p3);
    % lines
    \draw[-] (0,0) -- (0.7,0.1);
    \draw[-] (0.7,0.1) -- (1.5,0.3);
    \draw[-] (1.5,0.3) -- (2.3,0.6);
    \draw[-] (2.3,0.6) -- (8,4);
    
    % xxx
    \filldraw (6,4) circle[radius=1.5pt] node[above left, outer sep=2pt] (Y) {$(m+1,n+1)$};
    \filldraw (8,4) circle[radius=1.5pt] node[above, outer sep=2pt] (X) {$(2m-n+1,n+1)$};
    \end{tikzpicture}
    
We note that from our assumption that $\mathrm {codim}_{V_f} W_0 < {n+1 \over 2m-n+1} \alpha (W_0)$, $P = (\alpha(W_0), \mbox{codim}_{V_f} W_0)$ lies below
the line $y = {n+1\over 2m-n+1} x$. From $m$-subgeneral position, it also lies to the left of the
the line $y=x+n-m$. Therefore $P$ must lie below and to the left of the intersection point $Q = \left({2m-n+1\over 2} , {n+1\over 2}\right)$ of the above two straight lines. 
Thus, we have
  \begin{equation}\label{E1}
  \alpha(W_0)<{2m-n+1\over 2},
  \end{equation}
  $$\mbox{codim}_{V_f} W_0<{n+1\over 2}.$$
  Since $\sigma > {n+1\over 2m-n+1}$ (See the figure), by using (\ref{E1}), we obtain
   \begin{eqnarray*}
   	\alpha(W_0)+ {n+1\over \sigma} &<& \alpha(W_0)+ 2m-n+1\\
   	&<& {2m-n+1 \over 2} + 2m-n+1\\
   	&=&{3\over 2}(2m-n+1).
  \end{eqnarray*}
Combing above with (\ref{new}), we get
$$\alpha(W_0)+(n+1){\alpha(W)-\alpha(W\cup W_0)\over  \mbox{codim}_{V_f} W} \leq {3\over 2}(2m-n+1).$$
The theorem thus follows from Corollary \ref{cor}. \end{proof}


\begin{thebibliography}{Qua19}
     	
\bibitem{DT}      	
 G. Dethloff and T. V. Tran,
 \newblock A second main theorem for moving hypersurface targets.
\newblock {\em Houston J. Math}, {\bf 37}(1), 79-111, 2011. 	

\bibitem{HL} 
G. Heier and A. Levin, 
\newblock A Schmit-Nochka theorem for closed subschemes in subgeneral position.
\newblock {\em Preprint}.

\bibitem{Nev}      	
R. Nevanlinna,
\newblock Le th\'eor\`eme de Picard-Borel et la th\'eorie des fonctions m\'eromorphes.
\newblock {\em Gauthier-Villars, Paris}, 1929. 	

\bibitem{Nochka}      	
 Nochka, E. I., 
 \newblock On the theory of meromorphic curves. (Russian) 
 \newblock {\em Dokl. Akad. Nauk SSSR},  269(3), 547–552, 1983.
 
 \bibitem{CF}      	
C. F. Osgood,
\newblock Sometimes effective Thue-Siegel-Roth-Schmidt-Nevanlinna bounds, or better.
\newblock {\em J. Number Theory}, {\bf 21}(3), 347-389, 1985. 	

\bibitem{Qua22}      	
S. D. Quang,
\newblock Meromorphic mappings into projective varieties with arbitrary families of moving hypersurfaces.
\newblock {\em J. Geom. Anal.}, {\bf 32}(2), Paper No.52, 29pp, 2022. 	

\bibitem{ru1}
M. Ru,  \newblock A defect relation for holomorphic curves intersecting hypersurfaces.
\newblock {\em Amer. J. Math.}, 126(1), 215–226, 2004.

\bibitem{ru_annals}
M. Ru, 
\newblock Holomorphic curves into algebraic varieties.
\newblock {\em Ann. of Math.}, 169(1): 255--267, 2009.

\bibitem{book:minru}
M. Ru,
\newblock {\em Nevanlinna theory and its relation to Diophantine
  approximation}.
\newblock Second Edition, World Scientific, 2021.

\bibitem{RuS1}     M. Ru and  W. Stoll,  \newblock  The second main theorem for moving targets. \newblock {\em J. Geom. Anal.}, 1(2), 99–138, 1991.

\bibitem{RuS2}      	
M. Ru and W. Stoll,
\newblock The Cartan conjecture for moving targets.
\newblock In: Several complex variables and complex geometry, Proc. Summer Res. Inst., Santa Cruz/CA (USA) 1989, Proc. Symp. Pure Math. 52, Part 2, 477-508, 1991. 
     
\bibitem{SYY}
Lei Shi, Qiming Yan and Guangsheng Yu,
\newblock Second main theorems for holomorphic curves in the projective space with slowly moving hypersurfaces.
\newblock To apper.
     
\bibitem{Stein}      	
N. Steinmetz,
\newblock Eine verallgemeinerung des zweiten Nevanlinnaschen hauptsatzes.
\newblock {\em J. Reine Angew. Math.}, {\bf 368}, 134-141, 1986. 		     
     	
\bibitem{Yama}      	
K. Yamanoi,
\newblock The second main theorem for small functions and related problems.
\newblock {\em Acta Math.}, {\bf 192}(2), 225-294, 2004. 		          
     	
\bibitem{yan} 
Q. M. Yan and G. S. Yu,
\newblock Cartan's conjecture for moving hypersurfaces.
\newblock {\em Math. Z.}, {\bf 292}(3-4), 1052-1067, 2019. 

\bibitem{Voj87}
P. Vojta,
\newblock Diophantine approximations and value distribution theory, volume 1239 of {\em Lecture Notes in Mathematics.} Springer-Verlag, Berlin, 1987.

\bibitem{Voj07}
P. Vojta,
\newblock On the Nochka-Chen-Ru-Wong proof of Cartan's conjecture.
\newblock {\em J. Number Theory.}, {\bf 125}(1):229--234, 2007.


\end{thebibliography}
\end{document}